\documentclass[10pt, final]{amsart}
\usepackage[english]{babel}
\usepackage{latexsym}
\usepackage{amssymb}
\usepackage{amscd}
\usepackage[cp850]{inputenc}
\usepackage[mathscr]{eucal}

\theoremstyle{plain}

\newtheorem{theorem}{Theorem}
\newtheorem{proposition}{Proposition}

\newtheorem{cor}{Corollary}
\newtheorem{lema}{Lemma}
\theoremstyle{remark}

\newcommand{\R}{\mathbb{R}}
\newcommand{\N}{\mathbb{N}}

\def\Ext{\operatorname{Ext}}

\newcommand{\aproof}{\begin{proof}}
\newcommand{\zproof}{\end{proof}}

\begin{document}
\title{Nonseparable $C(K)$-spaces can be twisted\\ when $K$ is a finite height compact}

\author{Jes\'us M. F. Castillo}

\address{Departamento de Matem\'aticas\\ Universidad de Extremadura\\
Avenida de Elvas\\ 06011-Badajoz\\ Spain} \email{castillo@unex.es}

\thanks{Thanks are due to a demanding referee who pushed the author to look for nicer arguments and, in the end, to produce a better paper. Additional thanks are due to Daniel Tausk who made a few accurate remarks about the content of the paper.}

\thanks{This research was supported by project MTM2013-45643-C2-1-P, Spain.}

\maketitle

\begin{abstract} We show that, given a nonmetrizable compact space $K$ having $\omega$-derived set empty, there always exist nontrivial exact sequences $0\to c_0\to E\to C(K)\to 0$. This partially solves a problem posed in several papers: Is $\Ext(C(K), c_0)\neq 0$ for $K$ a nonmetrizable compact set?\end{abstract}

\section{Introduction}

A general problem in Banach space theory is to determine, given two Banach spaces $Y$ and $Z$, the existence and properties of Banach spaces $X$ containing $Y$ such that $X/Y=Z$. The space $X$ is called a twisted sum of $Y$ and $Z$. When only the trivial situation $X=Y\oplus Z$ is possible one obtains an interesting structure result asserting that every copy of $Y$ in a space $X$ such that $X/Y=Z$ must be complemented. When some $X\neq Y\oplus Z$ as above exists, one obtains a usually exotic and interesting space with unexpected properties; perhaps the perfect example could be the Kalton-Peck $Z_2$ space \cite{kaltpeck} which has no unconditional basis while containing an uncomplemented copy of $\ell_2$ such that $Z_2/\ell_2=\ell_2$. Recall that a short exact sequence of Banach spaces is a diagram \begin{equation}\label{SEX}
\begin{CD}
0  @>>>  Y @>i>> X @>q>> Z @>>>0
\end{CD}
\end{equation}
where $Y$, $X$ and $Z$ are Banach spaces and the arrows are operators in such a way that the kernel of each
arrow coincides with the image of the preceding one. By the open mapping theorem $i$ embeds $Y$ as a closed subspace of $X$ and $Z$ is
isomorphic to the quotient $X/i(Y)$.  Two exact sequences $0\to Y \to X \to Z
\to 0$ and $0\to Y \to X_1 \to Z \to 0$ are said to be
\emph{equivalent} if there exists an operator $T:X\longrightarrow
X_{1}$ making commutative the diagram
$$
\begin{CD}
0 @>>>Y@>>>X@>>>Z@>>>0\\
&&\| &&@VTVV \|\\
0 @>>>Y@>>>X_{1}@>>>Z@>>>0.\end{CD}$$ The sequence (\ref{SEX}) is said to be trivial (or that it splits) if $i(Y)$ is complemented in $X$ (i.e., if it is equivalent to the sequence $0\to Y\to Y\oplus Z\to Z\to 0$). We write $\Ext(Z,Y)=0$ to indicate that every sequence of the form (\ref{SEX}) splits. Summing all up, the general problem is to determine when there exist nontrivial twisted sums of $Y$ and $Z$, or else, when $\Ext(Z,Y) \neq 0$.

We pass now to consider the specific case in which both $Y$ and $Z$ are spaces of continuous functions on compact spaces.
The twisting of separable $C(K)$-spaces was treated and, to a
large extent, solved in \cite{ccky}. Nevertheless, the problem of
constructing nontrivial twisted sums of large $C(K)$-spaces is
mostly unsolved. The following problem was posed and considered in
\cite{ccky,ccy,cgpy}:  Is it true that for every non-metrizable
compact $K$ one has $\Ext(C(K),c_0) \neq 0$? In this paper we
prove that  $\Ext(C(K), C(S))\neq 0$ for any two compact spaces $K,
S$ with $\omega$-derived set empty and $K$ nonmetrizable which, in
particular, covers the case above for $C(S)=c_0$.

\section{Main result, and its consequences}

In what follows, the cardinal of a set $\Gamma$ will be denoted $|\Gamma|$. We will write $c_0(\Gamma)$ or $c_0(|\Gamma|)$ depending if one needs to stress the set or only the cardinal.

We begin with a couple of observations: First of all, that if $\Gamma$ is an uncountable set then $\Ext(c_0(\Gamma), c_0)\neq 0$. To explicitly construct an example it is enough to assume that $|\Gamma|=\aleph_1$, pick an isomorphic embedding $\varphi: c_0(\aleph_1)\to \ell_\infty/c_0$ and consider the commutative diagram
$$\begin{CD}
 0 @>>> c_0 @>>> \ell_\infty @>p>> \ell_\infty/c_0 @>>> 0\\
&& @| @AAA @AA{\varphi}A\\
 0 @>>> c_0 @>>> p^{-1}\left (\varphi(c_0(\aleph_1))\right)  @>p_{\varphi}>>c_0(\aleph_1) @>>> 0
\end{CD}
$$where $p_\varphi(x) = \varphi^{-1} p(x)$. Indeed, to see that the lower sequence does not split, one only has to observe that in any commutative diagram
\begin{equation}\label{isom}\begin{CD}
 0 @>>> c_0 @>>> \ell_\infty @>>> \ell_\infty/c_0 @>>> 0\\
&& @| @AA{\varphi'}A @AA{\varphi}A\\
 0 @>>> c_0 @>>> X  @>q>>c_0(\Gamma) @>>> 0.
\end{CD}\end{equation} in which $\varphi$ is an isomorphic embedding then also $\varphi'$ is an isomorphic embedding, which implies that the quotient map $q$ cannot be an isomorphism onto any non
separable subspace of $c_0(\Gamma)$ since non-separable $c_0(\Gamma')$ spaces are not subspaces of
$\ell_\infty$. The lower sequence in diagram (\ref{isom}) can however split when $\varphi$ is just an operator. Moreover, the results in \cite{castmoresing} imply that in any exact
sequence
$$\begin{CD} 0@>>> c_0 @>>> X@>q>> c_0(J) @>>>0\end{CD}$$
in which $|J|$ large enough the operator $q$ becomes an isomorphism on copies of
$c_0(I)$ for large $I$; while, on the other hand, in an exact sequence
$$\begin{CD} 0@>>> Y @>>> X@>q>> c_0(J) @>>>0,\end{CD}$$
in which $\Ext(X, c_0)=0$ then $q$ cannot be an isomorphism onto any
non-separable subspace of $c_0(J)$: otherwise, if $q$ becomes an
isomorphism on some subspace $c_0(J')$, this must be complemented
in $c_0(J)$ by \cite{grane}, and therefore $q^{-1}(c_0(J'))$ must be complemented in $X$, which prevents
$\Ext(X,c_0)=0$.  It is still open to decide whether every twisted sum of two $c_0(\Gamma)$ must be a $C(K)$-space
(see \cite{castsimo} for related information). We prove two preparatory lemmata of independent interest. The first one is a reformulation of \cite[Lemma 1]{lindrose}. Recall that given any exact sequence $0\to A\to B\to B/A\to 0$ and an operator $\phi: A\to Y$ there exists a superspace $X$ of $Y$ such that $X/Y=B/A$ (see \cite[Prop. 1.3.a]{castgonz}), thus forming a commutative diagram
\begin{equation}\label{pushing}
\begin{CD}
0 @>>> A @>>> B @>>> B/A@>>> 0\\
&&@V{\phi}VV @VVV @|\\
0 @>>> Y @>>> X @>>> B/A@>>> 0.
\end{CD}
\end{equation}

The following lemma shows that the converse is somehow true:

\begin{lema} Let $0\to A \to B\to C \to 0$ be an exact sequence and let $E$ be a Banach space. If $\Ext(B,E)=0$ then
for every exact sequence  $0\to E \to X\to C \to 0$ there is an operator $\phi: A\to E$ so that there is a commutative diagram
\begin{equation}
\begin{CD}
0 @>>> A @>>> B @>>> C@>>> 0\\
&&@V{\phi}VV @VVV @|\\
0 @>>> E @>>> X @>>> C@>>> 0.
\end{CD}
\end{equation}
\end{lema}

This result can be seen as a consequence of the homology sequence (see \cite{cabecastlong}), whose part relevant for us is the existence of an exact sequence:
$$
\begin{CD}
\dots @>>> \mathcal L (A, E) @>{\omega_E}>> \Ext(C, E)@>>> \Ext(B, E) @>>> \cdots.
\end{CD}
$$
Thus, if $\Ext(B, E)=0$ then the map $\omega_E: \mathcal L (A, E) \to \Ext(C, E)$ is surjective. This map $\omega_E$ sends operators $\phi: A\to E$ into the lower sequence in diagram (\ref{pushing}).

\begin{lema}\label{dualcardinal} Let \begin{equation}\label{twisted}\begin{CD}
 0 @>>> Y  @>>> X  @>>>c_0(\mathfrak c) @>>> 0
\end{CD}
\end{equation} be an exact sequence. If $|Y^*| \leq \mathfrak c$ then $\Ext(X, c_0)\neq 0$.
\end{lema}

\begin{proof} Applying the homology sequence as above to sequence (\ref{twisted}) with target $E=c_0$  one  gets that the map
$$\begin{CD}\omega_Y: \mathcal L(Y, c_0) @>>> \Ext(c_0(\mathfrak c), c_0)\end{CD}$$
is surjective when $\Ext(X, c_0)=0$, from where $|\Ext(c_0(\mathfrak c), c_0)|  \leq |\mathcal L (Y, c_0)|$. Let us show that, under the assumption  $|Y^*| \leq \mathfrak c$, this cannot be so: On one hand one has
$|\mathcal L(Y, c_0)|\leq |\mathcal L(\ell_1, Y^*)|$, which is the set of bounded sequences of
the unit ball of $Y^*$. Since there are $|(Y^*)^\N|$ countable
subsets of $Y^*$, each of them admitting $\mathfrak c$ bounded
sequences one gets, $$|\mathfrak L(Y, c_0)| \leq  |\R \times (|Y^*|^{\aleph_0})^{\aleph_0}| = \mathfrak c^{\aleph_0} = \mathfrak c.$$ On the other hand, Marciszewski and Pol show in \cite[7.4]{marcpol}
that there exist $2^{\mathfrak c}$ non-equivalent exact sequences $0\to c_0\to \spadesuit \to c_0(\mathfrak c)\to 0$; i.e., $|\Ext(c_0(\mathfrak c), c_0)|\geq 2^{\mathfrak c}$.\end{proof}

The reason behind the involved proof to come is that the straightforward proof one would obtain from the argument above does not holds in general since the inequality $\aleph^{\aleph_0} < 2^\aleph$ does not
necessarily hold for all cardinals. Indeed, see \cite[Thm. 5.15]{jech}, under {\sf GCH}, it occurs that when $\aleph$ has cofinality greater than $\aleph_0$ then $\aleph^{\aleph_0} = \aleph$; but for $\aleph$ having cofinality $\aleph_0$ then $\aleph^{\aleph_0} = 2^\aleph$.\\

Moreover, the Continuum Hipothesis, in what follows [{\sf CH}], is necessary to make $\mathfrak c$ the first case to consider. Indeed, the argument of Marciscewski-Pol could not work even for $\aleph_1$ outside {\sf CH} since it depends on cardinal arithmetics: for $\mathfrak c < 2^{\aleph_1}$ the same proof in \cite{marcpol}  yields that there are $2^{\aleph_1}$ different exact sequences $0\to c_0 \to X \to c_0(\aleph_1) \to 0$. However, if $2^{\aleph_1} = \mathfrak c$ then the method in \cite{marcpol} does not decide. We are indebted to W. Marciszewki who informed us from these facts. \\

Lemma \ref{dualcardinal} can be improved to:

\begin{lema}\label{retwisted} Let $\;0 \to  Y \to  X \to  Z \to  0$ be  an exact sequence. Let $B$ be a Banach space for which $|\mathcal L(Y, B)| < |\Ext(Z,B)|$.  Then $|\Ext(X, B)|\geq |\Ext(Z, B)|$.
\end{lema}
\begin{proof} It follows from  the exactness of the sequence
$$\begin{CD}\mathcal L(Y, B) @>{\omega_{B}}>> \Ext (Z, B)  @>\gamma>> \Ext(X, B) \end{CD}$$
and the standard factorization $\Ext (Z, B)/  \ker  \gamma = \mathrm{Im} \gamma$  that $|\Ext(X, B)|\geq | \mathrm{Im}\; \gamma| = |\Ext (Z, B)/ \; \ker \gamma| = |\Ext (Z, B)/ \; \mathrm{Im}\; \omega_{B}|$. On the other hand one has
$|\mathrm{Im}\; \omega_{B} | \leq |\mathcal L(Y, B)| < |\Ext (Z, B)| $, and thus $|\Ext (Z, B)/ \; \mathrm{Im}\; \omega_{B}| = |\Ext(Z,B)|$.\end{proof}

To continue with the preparation for the proof we need now a result that can be considered folklore but which, to the best of our knowledge, cannot be found in the literature:

\begin{lema}\label{dens} Let $Y$ be a subspace of $X$. There is a subspace $Y^\bullet\subset Y$ with density character $\mathrm{dens} \;Y^\bullet \leq \mathrm{dens} \; X/Y$ and a subspace $X^\bullet\subset X$ such that $X^\bullet/ Y^\bullet = X/Y$.\end{lema}
\begin{proof} Let $(z_\gamma)_\gamma$ be a dense subset of the unit ball of $X/Y$ having minimal size. By the open mapping theorem there is some constant $C>0$ such that one can pick  for each $z_\gamma$ an element $x_\gamma\in X$ with norm $\|x_\gamma\|\leq C\|z_\gamma\|$ so that $x_\gamma +Y = z_\gamma$. Let $L: X/Y\to X$ be a linear (not necessarily continuous) selection for the quotient map $q:X\to X/Y$. Thus $x_\gamma - Lz_\gamma \in Y$ and the closed subspace they span $Y^\bullet = [x_\gamma - Lz_\gamma]$ has density character smaller or equal than that of $Z$. Let $X^\bullet= [Y^\bullet + [x_\gamma]]$ be the closed subspace of $X$ spanned by $Y^\bullet$ and all the $x_\gamma$. It is clear that $q: X^\bullet \to X/Y$ is onto since the points $x_\gamma + Y =z_\gamma$ are in the image of the ball of radius $C$ of $X^\bullet$. Moreover, $\ker q_{|X^\bullet} = Y^\bullet$: indeed, since $X^\bullet = [Y^\bullet + [x_\gamma]] = [x_\gamma - Lz_\gamma, x_\gamma] = [x_\gamma, Lz_\gamma]$,
pick $x\in X^\bullet$ of the form $x=\sum \lambda_\gamma x_\gamma + \sum \lambda_\mu Lz_\mu$ so that $qx=0$; which means
$\sum \lambda_\gamma z_\gamma = - \sum \lambda_\mu z_\mu$. Thus

$$x=\sum \lambda_\gamma x_\gamma + \sum \lambda_\mu Lz_\mu =\sum \lambda_\gamma x_\gamma + L\left(\sum \lambda_\mu z_\mu\right) = \sum \lambda_\gamma x_\gamma - L\left( \sum \lambda_\gamma z_\gamma \right) = \sum \lambda_\gamma (x_\gamma -Lz_\gamma)$$
and thus $x\in Y^\bullet$.\end{proof}

Lemma \ref{dens} can be reformulated as: given an exact sequence $0\
\to Y\to X\to Z\to 0$ there are subspaces $Y^\bullet\subset Y$ and $X^\bullet\subset X$ with $\mathrm{dens} \;Y^\bullet \leq \mathrm{dens}\; Z$ that form a commutative diagram
$$\begin{CD}
0@>>>Y^\bullet@>>>X^\bullet@>>>Z@>>>0\\
&&@VVV @VVV @| \\
0@>>>Y @>>>X @>>>Z@>>>0.
\end{CD}$$

In particular, given an exact sequence $ 0 \to c_0(\Gamma) \to  X \to c_0(\mathfrak c) \to 0 $, since for every subspace $E\subset c_0(\Gamma)$ with $\mathrm{dens}\; E\leq \aleph \leq
|\Gamma|$ there is a copy of $c_0(\aleph)$ for which $E \subset
c_0(\aleph)\subset c_0(\Gamma)$, it is clear that there is a commutative diagram
$$
\begin{CD}
0 @>>> c_0(\mathfrak c)  @>>> X^\bullet @>>> c_0(\mathfrak c) @>>>0\\
&&@ViVV @VV{i^\bullet}V @|\\
0 @>>> c_0(\Gamma)  @>>> X @>>> c_0(\mathfrak c) @>>>0,
\end{CD}
$$where $i, i^\bullet$ are the canonical inclusions. We are ready for the action:\\

\noindent \textbf{Basic case: $|J|=\mathfrak c$:}
\begin{lema}\label{basic} Given an exact sequence $0\to c_0(\Gamma) \to X\to c_0(\mathfrak c) \to 0$ then $\Ext(X, c_0)\neq 0$.\end{lema}
\begin{proof}  If $\Ext(X, c_0)=0$ then  the map
$\omega_{c_0}: \mathcal L(c_0(\Gamma), c_0)\to \Ext(c_0(\mathfrak c), c_0)$ is surjective and thus, via composition with $i$, there is also a surjective map $\mathcal L(c_0(\mathfrak c), c_0)\to \Ext(c_0(\mathfrak c), c_0)$, which is impossible.\end{proof}

\noindent \textbf{Reduction to the case $|J|=\mathfrak c$:}
\begin{lema}\label{reduction} Under {\sf CH}, given an exact sequence $0\to c_0(\Gamma) \to X\to c_0(J) \to 0$ then $\Ext(X, c_0)\neq 0$ when $X$ is nonseparable.\end{lema}
\begin{proof} If $|J|\leq \aleph_0$ then the sequence splits (the spaces $c_0(\Gamma)$ are separably injective, which means that they are complemented in every superspace $X$ so that $X/c_0(\Gamma)$ is separable (see, e.g., \cite{accgm1}). In which case, $X$ is isomorphic to $c_0(\Gamma)$. Being $X$ nonseparable, $\Gamma$ must be
uncountable and $\Ext(X, c_0)\neq 0$. If $|J|=\mathfrak c$, we are in the case of Lemma \ref{basic}. If $|J|>\mathfrak c$, take a set of size $\mathfrak c$ in $J$ and consider the decomposition $c_0(J)=c_0(\mathfrak c)\oplus c_0(H)$ with canonical embedding $u: c_0(\mathfrak c)\to c_0(J)$. Take an arbitrary exact sequence $0\to c_0\to A\to c_0(\mathfrak c)\to 0$ and multiply on the right by $c_0(H)$ to get the exact sequence $0\to c_0\to A\oplus c_0(H)\to c_0(\mathfrak c) \oplus c_0(H)\to 0$. Form then the commutative diagram
$$
\begin{CD}
0 @>>> c_0(\mathfrak c)  @>>> q^{-1}(u(c_0(\mathfrak c)))^\bullet @>>> c_0(\mathfrak c) @>>>0\\
&&@ViVV @VV{i^\bullet}V @|\\
0 @>>> c_0(\Gamma)  @>>> q^{-1}(u(c_0(\mathfrak c))) @>q_u>> c_0(\mathfrak c @>>>0\\
&&@| @VV{u'}V @VVuV\\
0 @>>> c_0(\Gamma)  @>>> X @>q>> c_0(J) @>>>0\\
&&@V{\phi}VV @VV{\phi'}V @|\\
0 @>>> c_0  @>>>  A\oplus c_0(H) @>>> c_0(\mathfrak c)\oplus c_0(H) @>>>0.
\end{CD}
$$
where the operator $\phi$ exist because $\Ext(X, c_0)=0 $. If $p_A:  A\oplus c_0(H)\to A$ denotes the canonical projection onto $A$ then there is a commutative diagram
$$
\begin{CD}
0 @>>> c_0(\mathfrak c)  @>>> q^{-1}(u(c_0(\mathfrak c)))^\bullet @>>> c_0(\mathfrak c) @>>>0\\
&&@V\phi iVV @VV{p_A \phi' u' i^\bullet}V @|\\
0 @>>> c_0  @>>>  A @>>> c_0(\mathfrak c) @>>>0.
\end{CD}
$$This again implies that there is a surjective map $\mathcal L(c_0(\mathfrak c), c_0)\to \Ext (c_0(\mathfrak c), c_0)$, which we have already shown it is impossible.\end{proof}

We are now ready to show:

\begin{proposition}\label{induction} Let $X_n$ be a sequence of Banach spaces in which $X_0$ is a nonseparable twisted sum of $c_0(\Gamma_0)$ and $c_0(J_0)$, and $X_{n}$ is a twisted sum of $c_0(\Gamma_{n})$ and $X_{n-1}$ for all $n\geq 1$. Under {\sf CH}, $\Ext(X_{n}, c_0)\neq 0$ for all $n\geq 0$.
\end{proposition}
\begin{proof} We proceed (formally) inductively on $n$. That $\Ext(X_0, c_0)\neq 0$ has been already shown.\\

We show now that also $\Ext(X_1, c_0)\neq 0$. By the arguments preceding, there is no loss of generality assuming that
$|J_0|\geq \mathfrak c$, in which case there is also an exact sequence  $0\to c_0(\Gamma'_0)\to X_0 \to c_0(\mathfrak c)\to$ and then $X_1$ and $X_0$ are connected as in the commutative diagram
\begin{equation}\label{3arrow}
\begin{CD}
&&&& 0& &0\\
&&&&@AAA @AAA\\
&&&&c_0(\mathfrak c)&=&c_0(\mathfrak c)\\
&&&& @AAA @AApA\\
0 @>>> c_0(\Gamma_1)  @>>> X_1 @>q>> X_0 @>>>0\\
&&@| @AAA @AAuA\\
0 @>>> c_0(\Gamma_1)  @>>> q^{-1}(u(c_0(\Gamma'_0))) @>>{q_u} > c_0(\Gamma'_0) @>>>0\\
&&&&@AAA @AAA\\
&&&& 0& &0.\end{CD}
\end{equation}

Now, reasoning as in Lemma \ref{basic} and Lemma \ref{reduction} we can assume that $|\Gamma'_0| = \mathfrak c$, and thus that also $|\Gamma_1|\leq \mathfrak c$; in which case $|q^{-1}(u(c_0(\Gamma'_0)))^*|\leq \mathfrak c$ and Lemma \ref{dualcardinal} applies to yield $\Ext(X_1, c_0)\neq 0$.

The result for $X_2$ can be obtained via the same method: call $P_1=q^{-1}(u(c_0(\Gamma'_0)))$ from diagram (\ref{3arrow}). We get the diagram
\begin{equation}\label{4arrow}
\begin{CD}
&&&& 0& &0\\
&&&&@AAA @AAA\\
&&&&c_0(\mathfrak c)&=&c_0(\mathfrak c)\\
&&&& @AAA @AApA\\
0 @>>> c_0(\Gamma_2)  @>>> X_2@>q>> X_1 @>>>0\\
&&@| @AAA @AAuA\\
0 @>>> c_0(\Gamma_2)  @>>> q^{-1}(u(P_1)) @>>{q_u} > P_1 @>>>0\\
&&&&@AAA @AAA\\
&&&& 0& &0.\end{CD}
\end{equation}

After the appropriate size reduction we can assume that $|P_1^*|\leq \mathfrak c$ and $|\Gamma_2|\leq \mathfrak c$  which implies that $|q^{-1}(u(c_0(\Gamma_1)))^*|\leq \mathfrak c$ and thus Lemma \label{dualcardinal} applies to yield $\Ext(X_2, c_0)\neq 0$.

Call now $P_{n+1}= q^{-1}(u(P_n))$ for $n=1,2,\dots$ and iterate the argument.\end{proof}

Let us recall that given a compact set $K$, the derived set $K^{(1)}$ is the set of its accumulation points. If one calls $K^{(n+1)} = (K^{(n)})'$, the $\omega$-derived set is $K^\omega = \cap_n K^{(n)}$.


\begin{theorem} \emph{[{\sf CH}]} Given two compact spaces $K, S$ with
$\omega$-derived set empty and  $K$ nonmetrizable then $\Ext(C(K),
C(S))\neq 0$.
\end{theorem}
\begin{proof} The result is clear after recalling a few  well-known facts:
\begin{itemize}
\item That $C(S)$ spaces in which $S$ is a compact with
$S^\omega=\emptyset$ contain $c_0$ complemented as it follows from the existence of a convergent sequence in $S$.
\item If $I_K$ is  the set of isolated
points of $K$ the restriction operator $C(K)\to C(K^{(1)})$ induces a
short exact sequence $$\begin{CD} 0@>>> c_0(I_K)@>>> C(K)@>>> C(K^{(1)})@>>> 0.\end{CD}
$$\item That, after a finite number $N+2$ of derivations of $K$, one arrives to a finite compact, obtaining thus an exact sequence$$\begin{CD} 0@>>> c_0(I_{K^{(N)}})@>>> C(K^{(N)})@>>> c_0(J)@>>> 0.\end{CD}
$$ \end{itemize}
We are therefore in the hypotheses of Proposition \ref{induction} and thus $\Ext(C(K), c_0)\neq 0$ and therefore
 $\Ext(C(K), C(S))\neq 0$.
\end{proof}

Godefroy, Kalton and Lancien show in \cite[Thm. 2.2]{gkl} that a $C(K)$ space is Lipschitz isomorphic to $c_0(\N)$ if and only if it is linearly isomorphic to $c_0$. They also show that the result is no longer true in the nonseparable case, showing that a $C(K)$ space is Lipschitz isomorphic to some $c_0(I)$ if and only if the
$\omega$ derived set of $K$ is empty. Therefore,

\begin{cor} \emph{[{\sf CH}]}. If $C(K)$ and $C(S)$ are Lipschitz isomorphic to some (probably different) spaces $c_0(\Gamma)$ then
$\Ext(C(K), C(S))\neq 0$.
\end{cor}

For the sake of completeness, let us gather in an omnibus theorem
all known results about the problem considered. Recall that a compact space $K$ is said to have \emph{countable chain condition} ($ccc$ in short) when every family of pairwise disjoint open sets is countable. It is well known \cite[Thm. 4.5]{roseacta} that $K$ has \emph{ccc} if and only if $C(K)$ does not contain $c_0(I)$ for $I$ uncountable.

\begin{proposition} One has $\Ext(C(S),c_0) \neq 0$ whenever $C(S)$ contains a complemented non-separable $C(K)$ space for which
at least one of the following conditions hold:
\begin{enumerate}

\item $K$ is an Eberlein compact.

\item $K$ is a Valdivia compact without ccc.


\item \emph{[{\sf CH}]}  $K$ has finite height.

\item $C(K)$ admits a continuous injection into $C(\N^*)$ but not into $\ell_\infty$.

\item $K$ is an ordinal space.

\item $C(K)$ contains $\ell_\infty$.

\end{enumerate}
\end{proposition}
\begin{proof} The proof of (1) appears in \cite{cgpy}: According to Reif
\cite{reif}, since the space $C(K)$ is WCG it admits a
Marku\~{s}evi\v{c} basis $(x_\gamma, f_\gamma)_\gamma \in \Gamma$,
for which it can be assumed that $(f_\gamma)$ is bounded. It is
possible to define a dense-range operator $T: C(K) \to
c_0(\Gamma)$ as $T(f) = (f_\gamma(f))$. On the other hand, in any nontrivial exact sequences $0\to c_0\to X \stackrel{q}\to c_0(\Gamma)\to 0$ the space $X$ cannot be WCG (Weakly Compactly Generated) since $c_0$ is complemented in WCG spaces. Thus, one has a commutative diagram
$$
\begin{CD}
0@>>> c_0 @>>> X @>q>> c_0(\Gamma)@>>>0\\
&& @| @AtAA @AATA\\
 0@>>> c_0@>>> P_T @>>{\pi}> C(K) @>>> 0
\end{CD}$$ where $P_T=\{(f,g)\in X \oplus_\infty C(K): q(f)=T(g)\}$, $t(f,g)=f$ and $\pi(f,g)=g$ (indeed, when $T$ is not an embedding one cannot expect the middle operator $t$ to be an embedding; i.e., one cannot expect the lower twisted sum space to be a subspace of the upper twisted sum space). The ``dense range" version of the 3-lemma implies that the operator
$t$ has dense range; hence, the lower sequence cannot split since otherwise $P_T \simeq c_0\oplus C(K)$ would be
WCG, as well as $X$.\\

(2) follows from \cite[Thm. 1.2]{acgjm}: if $K$ is a Valdivia
compact and $c_0(I) \subset C(K)$ then there is a subset $J
\subset I$ such that $|J|=|I|$ and $c_0(J)$ is complemented in
$C(K)$.

Assertion (3) is contained in the main results of the paper.

Assertion (4) follows from the commutative diagram
$$
\begin{CD}
0@>>> c_0 @>>> \ell_\infty @>q>> C(\N^*)@>>>0\\
&& @|  @AtAA @AA{T}A\\
 0@>>> c_0@>>> q^{-1}(\varphi(C(K))) @>{q_{T}}>> C(K) @>>> 0
\end{CD}
$$
in which $T$ is the continuous injection claimed in he hypothesis. If the lower sequence splits, $t_{|C(K)}$ would be a continuous injection $C(K)\to \ell_\infty$.

Instances of the situation just described appear when, say, $K$ contains a dense set of weight at most $\aleph_1$ but $C(K)$ is not a subspace of $\ell_\infty$, as it follows from Parovi\v{c}enko's theorem \cite{paro}, which implies the existence of a continuous surjection $\varphi: \N^* \to
K$, which in turn provides the injection $\varphi^\circ: C(K)\to C(\N^*)$. See also \cite[Prop. 0.1]{yostjl}. Or else, $C(K)$ spaces with non weak*-separable dual, but admitting continuous injections into $C(\N^*)$.

(5) The case of uncountable ordinal spaces can be reduced to $K=[0,\omega_1]$, and this is contained in (4).

(6) is obvious since $\Ext(\ell_\infty, c_0) \neq 0$ (see \cite{cabecastuni}) and
$\ell_\infty$ will be complemented in $C(K)$.\end{proof}

In \cite[p. 4539-4540]{ccky} we claimed that the argument in (1) can be extended to cover the case of Corson compacta. We have been informed by Daniel Tausk that such result does not immediately follows with the same arguments: on one side, consider any nontrivial exact sequence $0\to c_0\to X\to c_0(I)\to 0$. This prevents $B_{X^*}$ from being a Valdivia compact. If we assume $K$ is a Corson compact then from \cite[Thm. 1.7]{amn} it follows that there is an injective operator $T: C(K) \to c_0(I)$,  or a dense-range
operator $T:C(K)\to c_0(I)$ if one uses either  Markushevich basis for $C(K)$ or increasing PRI \cite{sinha}, argument also valid for $K$ Valdivia. Thus, one gets a commutative diagram
$$
\begin{CD}
0@>>> c_0 @>>> X @>q>> c_0(\Gamma)@>>>0\\
&& @| @AtAA @AATA\\
 0@>>> c_0@>>> P_T @>>> C(K) @>>> 0.
\end{CD}$$
But some argument is required now to conclude that no continuous lifting $C(K) \to X$ exists in order to make the lower sequence nontrivial. If $T$ has been chosen having dense range, then $t^*: X^* \to C(K)^*$ is an imbedding; thus, if $B_{C(K)^*}$ is a Corson compact then also $B_{X^*}$ would be a Corson compact, which is a contradiction.  However, it is consistent with {\sf ZFC} the existence of Corson compacta $K$ so that the unit ball of $C(K)^*$ is not Corson (see \cite[Thm. 5.9]{negre}, or else \cite[Thm.3.5, Ex.3.10]{amn}). The situation for Valdivia compacta is not much better: while now $ K$ Valdivia implies that $B_{C(K)*}$ is Valdivia \cite[Thm. 5.2]{kale}, it is no longer true that closed
subspaces of Valdivia compacta are Valdivia. Therefore, it is worth to state the problem:\\

\textbf{Problem.} Assume that $K$ is a (nonmetrizable) Corson or Valdivia compact. Does one have $\Ext(C(K), c_0)\neq 0$?\\

\emph{Added in proof.} In the meantime Correa and Tausk have shown in \cite{cota} that under {\sf CH} all Corson compact admit a nontrivial twisting with $c_0$ and that the same holds for ``most" Valvidia compact.

\end{document}